\documentclass[11pt,reqno]{amsart}

\usepackage{xcolor}
\usepackage[T1]{fontenc}
\usepackage{amsmath}							
\usepackage{amssymb}
\usepackage{amsthm}
\usepackage{amscd}
\usepackage{amsfonts}
\usepackage{mathabx}
\usepackage{stmaryrd}
\usepackage[all]{xy}

\usepackage{caption}
\usepackage{subcaption}
\usepackage{euler}

\usepackage{extarrows}

\usepackage[colorlinks, linktocpage, citecolor = purple, linkcolor = blue]{hyperref}
\usepackage{color}
\usepackage{tikz}	
\usetikzlibrary{cd}
\usetikzlibrary{matrix}
\usetikzlibrary{patterns}
\usetikzlibrary{positioning}
\usetikzlibrary{decorations.pathmorphing}
\usetikzlibrary{decorations.markings}
\usepackage{hypcap}

\usepackage{ytableau}
\usepackage{fullpage}
\usepackage[shortlabels]{enumitem}

\newtheorem*{theorem_1*}{Main Theorem}
\newtheorem*{theorem_2*}{Theorem $A$}
\newtheorem*{theorem_3*}{Theorem $B$}
\newtheorem{theorem}{Theorem}[section]
\newtheorem{lemma}[theorem]{Lemma}
\newtheorem{proposition}[theorem]{Proposition}

\theoremstyle{definition}
			\newtheorem{example}[theorem]{Example}
				
\newtheorem{definition}[theorem]{Definition}
\newtheorem{remark}[theorem]{Remark}
\newtheorem*{remark*}{Remark}

\newcommand{\ZZ}{\mathbb{Z}}

\newcommand{\QQ}{\mathbb{Q}}
\newcommand{\RR}{\mathbb{R}}

\newcommand{\te}{\widetilde{e}}
\newcommand{\tf}{\widetilde{f}}

\newcommand{\thh}{\widetilde{h}}

\newcommand{\tv}{\widetilde{v}}

\newcommand{\tG}{\widetilde{G}}

\newcommand{\tcalB}{\widetilde{\calB}}

\newcommand{\al}{\alpha}

\newcommand{\ga}{\gamma}
\newcommand{\de}{\delta}

\newcommand{\Ga}{\Gamma}

\newcommand{\tal}{\widetilde{\alpha}}
\newcommand{\tbe}{\widetilde{\beta}}
\newcommand{\tga}{\widetilde{\gamma}}
\newcommand{\tde}{\widetilde{\delta}}

\newcommand{\tGa}{\widetilde{\Gamma}}

\DeclareMathOperator{\dil}{dil}

\DeclareMathOperator{\Ker}{Ker}
\DeclareMathOperator{\co}{Coker}
\DeclareMathOperator{\Hom}{Hom}
\DeclareMathOperator{\Sym}{Sym}

\DeclareMathOperator{\Vol}{Vol}
\let\Pr\relax
\DeclareMathOperator{\Pr}{Pr}

\newcommand{\calB}{\mathcal{B}}

\DeclareMathOperator{\Id}{Id}

\DeclareMathOperator{\Nm}{Nm}

\DeclareMathOperator{\Prym}{Prym}
\DeclareMathOperator{\Gram}{Gram}

\DeclareMathOperator{\Jac}{Jac}

\newcommand{\Dmitry}[1]{{\color{blue}{\texttt Dmitry: #1}}}

\title[The Prym variety of a dilated double cover of metric graphs]{The Prym variety of a dilated double cover of metric graphs}

 \author{Arkabrata Ghosh}
 \address{}
 \email{\href{mailto:arka2686@gmail.com}{arka2686@gmail.com}}
 
   \author{Dmitry Zakharov}
  \address{Department of Mathematics, Central Michigan University, Mount Pleasant, MI 48859, USA}
 \email{\href{mailto:dvzakharov@gmail.com}{dvzakharov@gmail.com}}


\begin{document}

\maketitle

\begin{abstract}
We calculate the volume of the tropical Prym variety of a harmonic double cover of metric graphs having non-trivial dilation. We show that the tropical Prym variety behaves discontinuously under deformations of the double cover that change the number of connected components of the dilation subgraph.
\end{abstract}

\section{Introduction}

Tropical geometry studies discrete, piecewise-linear analogues of algebro-geometric objects. For example, the tropical analogue of an algebraic curve is a connected finite graph $G$, and the \emph{Jacobian} $\Jac(G)$ (also known as the \emph{critical group} of $G$) is a finite abelian group. We can endow $G$ with positive real edge lengths to obtain a \emph{metric graph} $\Ga$; this promotes the Jacobian to a real torus $\Jac(\Ga)$ equipped with an additional integral structure. The dimension of $\Jac(\Ga)$ is equal to the first Betti number $g(\Ga)=b_1(\Ga)$ of $\Ga$, also known as the \emph{genus} of $\Ga$.

The tropical analogue of a morphism of algebraic curves is a harmonic morphism of graphs. Topological covering spaces are examples of harmonic morphisms. More generally, harmonic morphisms of finite graphs allow nontrivial degrees at vertices, also known as \emph{dilation}, while harmonic morphisms of metric graphs also allow dilation along edges. A harmonic morphism of metric graphs is \emph{free} if it has no dilation (in other words, if it is a covering isometry), and \emph{dilated} otherwise. A harmonic morphism has a well-defined global degree, and a \emph{double cover} is a harmonic morphism of metric graphs of degree two. 

A classical algebraic construction associates to an \'etale degree two cover $\widetilde{X}\to X$ of smooth algebraic curves a principally polarized abelian variety (ppav) of dimension $g(X)-1$, called the \emph{Prym variety} $\Prym(\widetilde{X}/X)$ of the double cover (see~\cite{MR1010103}). The kernel of the norm map $\Nm:\Jac(\widetilde{X})\to \Jac(X)$ has two connected components, and $\Prym(\widetilde{X}/X)$ is the even connected component (containing the identity). The ppav $\Prym(\widetilde{X}/X)$ carries a principal polarization that is half of the polarization induced from $\Jac(\widetilde{X})$.

The tropical Prym variety was defined in~\cite{MR3782424} and further investigated in~\cite{MR4261102}, in complete analogy with the algebraic setting. A double cover $\pi:\tGa\to \Ga$ of metric graphs induces a norm map $\Nm:\Jac(\tGa)\to \Jac(\Ga)$. It is shown in~\cite{MR3782424, MR4261102} that the kernel of $\Nm$ has two connected components if $\pi$ is free and one if $\pi$ is dilated. In the free case, the connected component of the identity carries a principal polarization that is half the induced polarization, just as in the algebraic setting. In the dilated case, the kernel also carries a principal polarization, whose relationship to the induced polarization was computed in~\cite{rohrle2022tropical}. In either case, the corresponding principally polarized tropical abelian variety is called the \emph{Prym variety} $\Prym(\tGa/\Ga)$ of the double cover $\pi:\tGa\to \Ga$. 

\subsection*{The volume formulas} Kirchhoff's matrix tree theorem states that the order of the Jacobian group $\Jac(G)$ of a finite graph $G$ is equal to the number of its spanning trees. A weighted version of this result for metric graphs was proved in~\cite{MR3264262}: the square of the volume of the Jacobian $\Jac(\Ga)$ of a metric graph $\Ga$ of genus $g$ is given by a degree $g$ polynomial in the edge lengths of $\Ga$, whose monomials correspond to the spanning trees. Specifically, 
\[
\Vol^2(\Jac(\Ga))=\sum_{C\subset E(\Ga)} \prod_{e\in C}\ell(e),
\]
where the sum is taken over all $g$-element sets of edges of $\Ga$ that are complements of spanning trees.

An analogous formula for the volume of the Prym variety of a free double cover was proved in~\cite{MR4382460}. Let $\pi:\tGa\to \Ga$ be a free double cover of metric graphs of genera $2g-1$ and $g$, respectively, so that $\Prym(\tGa/\Ga)$ has dimension $g-1$. By analogy with~\cite{MR3264262}, the square of the volume of the Prym should be given by a polynomial with terms indexed by certain $(g-1)$-element sets of edges of $\Ga$, which play the role of spanning trees for double covers. The relevant notion was already defined by Zaslavsky in~\cite{zaslavsky1982signed} in the context of signed graphs, and was rediscovered by the authors of~\cite{MR4382460}, who were not aware of this earlier work. A set $F\subset E(\Ga)$ of $g-1$ edges is an \emph{odd genus one decomposition} (or \emph{relative spanning tree}) of \emph{rank} $r(F)$ if $\Ga\backslash F$ consists of $r(F)$ connected components, each having connected preimage in $\tGa$. It then turns out that the volume is given by a sum
\[
\Vol^{2}(\Prym(\tGa/\Ga))=\sum_{F\subset E(\Ga)} 4^{r(F)-1} \prod_{e\in F}\ell(e),
\]
over all genus one decompositions $F\subset E(\Ga)$, with each term additionally weighted according to the rank $r(F)$.

\subsection*{Our results}
In this paper, we generalize the volume formula to dilated double covers. As a first step, we extend the definition of odd genus one decompositions, or \emph{ogods}, to the dilated case (see Definition~\ref{def:ogod}). The volume of the Prym variety is then given by a similar formula.

\begin{theorem}(Theorem~\ref{thm:main})
Let $\pi:\tGa\to\Ga$ be a dilated double cover of metric graphs. The volume of the tropical Prym variety of $\pi:\tGa\to\Ga$ is given by
\begin{equation}
\Vol^2(\Prym(\tGa/\Ga))=2^{1-d(\tGa/\Ga)}\sum_{F\subset E(\Ga)} 4^{r(F)-1}\prod_{e\in F}\ell(e).
\label{eq:mainformulaintro}
\end{equation}
The sum is taken over all $h$-element ogods of $E(\Ga)$ (see Definition~\ref{def:ogod}), where $h=g(\tGa)-g(\Ga)$ is the dimension of $\Prym(\tGa/\Ga)$ and $r(F)$ is the rank of an ogod, and $d(\tGa/\Ga)$ is the number of connected components of the dilation subgraph $\Ga_{\dil}$ of $\Ga$.

\end{theorem}

The theorem is proved by performing a series of deformations to the base curve $\Ga$ in such a way that the double cover becomes free, and then applying the results of~\cite{MR4382460} (see Theorem~\ref{thm:A}). At the same time, it is necessary to explicitly compute the pushforward and pullback maps on the simplicial homology groups $H_1(\tGa,\ZZ)$ and $H_1(\Ga,\ZZ)$ for a dilated double cover; these calculations appeared in~\cite{rohrle2022tropical} and we restate them here for convenience. We then compute the relationship between the volumes of the three tropical ppavs attached to a dilated double cover $\pi:\tGa\to\Ga$ (see Theorem~\ref{thm:B}). Putting these two theorems together, we obtain our main result. 

We note that Equation~\eqref{eq:mainformulaintro} shows that the volume of the Prym variety does not behave continuously under deformations of the cover that change the number $d(\tGa/\Ga)$ of connected components of the dilation subgraph (see Example~\ref{ex:disc}). In other words, the object $\Prym(\tGa/\Ga)$ does not behave well in moduli, suggesting that perhaps it needs to be somehow redefined, at least in the dilated case. 

As a final note, we observe that the ogods of a dilated double cover $\pi:\tGa\to\Ga$ appear to form the bases of a matroid on the set of undilated edges of $\Ga$, generalizing the signed graphic matroid introduced by Zaslavsky in~\cite{zaslavsky1982signed}. We are not aware if this matroid has been considered before, and we plan to study it in future work.

\subsection*{Acknowledgements} The authors would like to thank Yoav Len and Felix R\"ohrle for insightful discussions.

\section{Setup and notation}

In this section, we recall a number of standard definitions concerning metric graphs, harmonic morphisms, double covers, and tropical abelian varieties. We also recall the Jacobian $\Jac(\Ga)$ of a metric graph $\Ga$ and the Prym variety $\Prym(\tGa/\Ga)$ of a harmonic double cover $\tGa\to\Ga$ of metric graphs. Finally, we recall how to compute the volumes of $\Jac(\Ga)$ (see~\cite{MR3264262}) and $\Prym(\tGa/\Ga)$ for a free double cover $\tGa\to\Ga$ (see~\cite{MR4382460}). 

\subsection{Graphs, metric graphs, and double covers}

A \emph{graph} $G$ consists of a non-empty finite set of \emph{vertices} $V(G)$ and a set of \emph {edges} $E(G)$. We allow loops and multiple edges between a pair of vertices. It is convenient to view each edge $e\in E(G)$ as consisting of two \emph{half-edges} $e=\{h,h'\}$, so that $E(G)$ is the set of orbits of a fixed-point-free involution acting on the set of half-edges $H(G)$. The \emph{root map} $r:H(G)\to V(G)$ attaches half-edges to vertices, and the set of half-edges $T_v(G)=r^{-1}(v)$ attached to a given vertex $v$ is the \emph{tangent space}. The \emph{genus} of a graph is its first Betti number (we do not use vertex weights):
\[
g(G)=b_1(G)=\#E(G)-\#V(G)+1.
\]
An \emph{orientation} of $G$ is a choice of ordering of each edge, and defines \emph{source} and \emph{target} maps $s,t:E(G)\to V(G)$.

A \emph{morphism of graphs} $f:\tG\to G$ is a pair of maps $f:V(\tG)\to V(G)$ and $f:H(\tG)\to H(G)$ that commute with the root map and preserve edges. A \emph{harmonic morphism of graphs} $(f,d_f):\tG\to G$ is a pair consisting of a morphism of graphs $f:\tG\to G$ and a \emph{degree function} $d_f:V(\tG)\cup E(\tG)\to \ZZ_{>0}$ satisfying
\[
d_f(\tv)=\sum_{\thh\in T_{\tv}\tG\cap f^{-1}(h)}d_f(\thh)
\]
for any $\tv\in V(\tG)$ and any $h\in T_{f(\tv)}G$ (where we denote $d_f(\thh)=d_f(\thh')=d_f(\te)$ for an edge $\te=\{\thh,\thh'\}$). If $G$ is connected, a harmonic morphism has a \emph{global degree} $\deg(f)$ given by
\[
\deg(f)=\sum_{\tv\in f^{-1}(v)}d_f(\tv)=\sum_{\thh\in f^{-1}(h)}d_f(\thh)
\]
for any $v\in V(G)$ or any $h\in H(G)$.

A \emph{double cover} $p:\tG\to G$ is a harmonic morphism of global degree two. There are two types of vertices $v\in V(G)$: \emph{undilated}, having two pre-images $p^{-1}(v)=\{\tv^+,\tv^-\}$ with $d_p(\tv^{\pm})=1$, and \emph{dilated}, having a single pre-image $p^{-1}(v)=\{\tv\}$ with $d_p(\tv)=2$. We similarly define dilated and undilated half-edges and edges of $G$. A dilated half-edge may only be rooted at a dilated vertex, hence the set of dilated edges and vertices forms a subgraph called the \emph{dilation subgraph} $G_{\dil}\subset G$ of $f$. We say that $f$ is \emph{free} if $G_{\dil}$ is empty, \emph{dilated} if $G_{\dil}$ is not empty, and \emph{edge-free} if $G_{\dil}$ consists of isolated vertices only. We note that dilated double covers should not be thought of as somehow "more degenerate" than free double covers; in fact, arguably the converse is true, since the latter arise as tropicalizations of more degenerate algebraic double covers than the former.


A \emph{metric graph} $\Ga$ is a compact metric space obtained from a graph $G$ by identifying each edge $e\in E(G)$ with a closed interval of positive length $\ell(e)$, identifying the endpoints with the vertices of $G$, and endowing $\Ga$ with the shortest-path metric. The pair $(G,\ell)$ is called a \emph{model} for $\Ga$. The \emph{genus} $g(\Ga)$ of a metric graph is its first Betti number, and is equal to the genus of any model. A \emph{harmonic morphism of metric graphs} $\phi:\tGa\to \Ga$ is a continuous, piecewise-linear map with nonzero integer-valued slopes $d_{\phi}(\te)=d_f(\te)$ along the edges $\te\in E(\tGa)$ given by the degree function $d_f$ of a harmonic morphism $f:\tG\to G$ between some models $(\tG,\widetilde{\ell})$ and $(G,\ell)$ of $\tGa$ and $\Ga$, respectively. The slope condition imposes the restriction
\[
\ell(\phi(\te))=d_{\phi}(\te) \widetilde{\ell}(\te),\quad \te\in E(\tG).
\]
We similarly define double covers of metric graphs. A double cover $\pi:\tGa\to\Ga$ is locally an isometry along an undilated edge and a factor 2 dilation along a dilated edge, which explains the terminology.

Finally, we recall how to contract a harmonic morphism along a subgraph of the target. Let $f:\tG\to G$ be a harmonic morphism and let $G_0\subset G$ be a possibly disconnected subgraph. We define $G'$ from $G$ by contracting each connected component of $G_0$ to a separate vertex. We similarly define $\tG'$ by contracting each connected component of $f^{-1}(G_0)$ to a separate vertex. Finally, we define $f':\tG'\to G'$ by setting $d_{f'}(\tv)$ on a vertex $\tv\in V(\tG')$ corresponding to a contracted component of $f^{-1}(G_0)$ to be equal to the global degree of $f$ on that component. The result is a harmonic morphism $f':\tG'\to G'$, called the \emph{contraction of $f$ along $G_0$}. We similarly define contractions of harmonic morphisms of metric graphs. 

Since we are principally interested in double covers of metric graphs, we henceforth assume that any metric graph $\Ga$ comes equipped with a choice of model. Hence we abuse notation and write $E(\Ga)$, $V(\Ga)$, and so on for a metric graph $\Ga$. The principal results of our paper do not depend on the choice of underlying model.

\begin{example}
\label{ex:main}

\noindent Figure~\ref{fig:example1} shows a dilated double cover $\pi: \tGa \to \Ga$ of metric graphs. Fat edges and vertices indicate dilation.

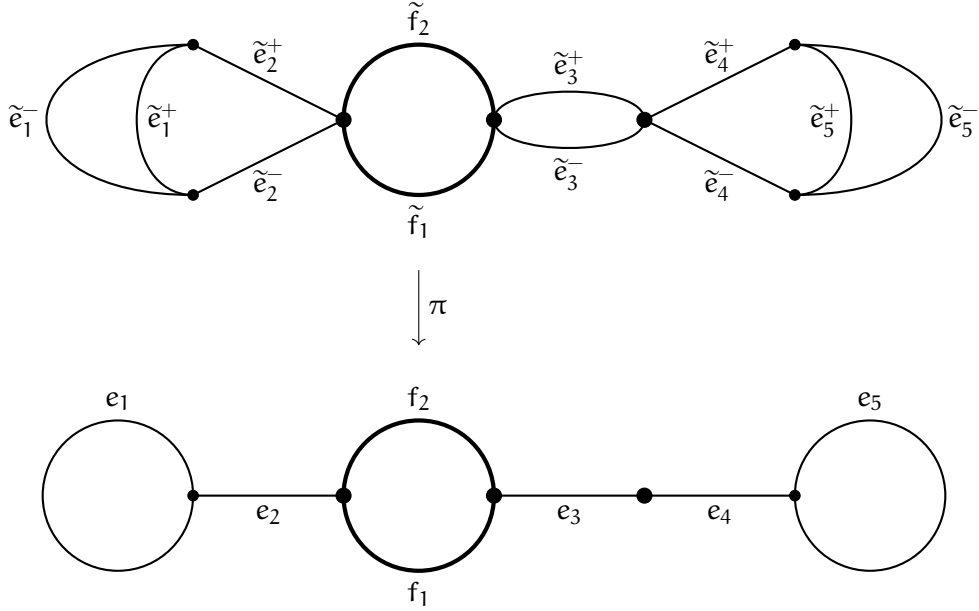
\begin{figure}

\begin{tikzpicture}

\draw[thick] (-4,1) .. controls (-6.6,1) and (-6.6,-1) .. (-4,-1);
\draw[thick] (-4,1) .. controls (-5,1) and (-5,-1) .. (-4,-1);
\node[right] at (-4.75,0) {$\te^+_1$};
\node[left] at (-5.9,0) {$\te^-_1$};
\draw[fill](-4,1) circle(0.07);
\draw[fill](-4,-1) circle(0.07);
\draw[thick] (-4,1) to (-2,0);
\draw[thick] (-4,-1) to (-2,0);
    \node[above] at (-3,0.5) {$\te^+_2$};
    \node[below] at (-3,-0.5) {$\te^-_2$};
\draw[ultra thick] (-1,0) circle(1);
    \node[above] at (-1,1) {$\widetilde{f}_2$};
    \node[below] at (-1,-1) {$\widetilde{f}_1$};

\draw[fill](-2,0) circle(0.1);
\draw[fill](0,0) circle(0.1);
\draw[thick] (0,0) .. controls (0,0.5) and (2,0.5) .. (2,0);
\draw[thick] (0,0) .. controls (0,-0.5) and (2,-0.5) .. (2,0);

\draw[fill](2,0) circle(0.1);

    \node[above] at (1,0.35) {$\te^+_3$};
    \node[below] at (1,-0.35) {$\te^-_3$};

\draw[thick] (4,1) .. controls (6.6,1) and (6.6,-1) .. (4,-1);
\draw[thick] (4,1) .. controls (5,1) and (5,-1) .. (4,-1);
\node[left] at (4.75,0) {$\te^+_5$};
\node[right] at (5.9,0) {$\te^-_5$};
\draw[fill](4,1) circle(0.07);
\draw[fill](4,-1) circle(0.07);
\draw[thick] (4,1) to (2,0);
\draw[thick] (4,-1) to (2,0);
    \node[above] at (3,0.5) {$\te^+_4$};
    \node[below] at (3,-0.5) {$\te^-_4$};

\path[->] (-1,-2) edge node[right]{$\pi$} (-1,-3);

\begin{scope}[yshift=-5cm]

\draw[thick] (-5,0) circle(1);
\node[above] at (-5,1) {$e_1$};%
\draw[fill](-4,0) circle(0.07);
\draw[thick] (-4,0) -- (-2,0);
\node[below] at (-3,0) {$e_2$};%
\draw[fill](-2,0) circle(.10);
\draw[fill](0,0) circle(.10);


\draw[ultra thick] (-1,0) circle(1);
\node[above] at (-1,1) {$f_2$};
\node[below] at (-1,-1) {$f_1$};

\draw[thick] (0,0) -- (2,0);
\node[below] at (1,0) {$e_3$};%
\draw[fill](2,0) circle(.10);
\draw[thick] (2,0) -- (4,0);
\node[below] at (3,0) {$e_4$};%
\draw[fill](4,0) circle(0.07);
\draw[thick] (5,0) circle(1);
\node[above] at (5,1) {$e_5$};%
\end{scope}
\end{tikzpicture}
\captionof{figure}{A dilated double cover} 
\label{fig:example1}
\end{figure}

\end{example}

\subsection{Tropical abelian varieties} The tropical Jacobian of a metric graph and the tropical Prym variety of a harmonic double cover are examples of tropical principally polarized abelian varieties, which are real tori equipped with auxiliary integral structure. We recall their definition (our conventions follow~\cite{MR4382460} and are a slight modification of the standard definitions found in~\cite{foster2018non} and~\cite{MR4261102}).

A \emph{real torus with integral structure} $\Sigma$ of dimension $n$, or \emph{integral torus} for short, is determined by a triple $(\Lambda,\Lambda',[\cdot,\cdot])$, where $\Lambda$ and $\Lambda'$ are finitely generated free abelian groups of rank $n$ and $[\cdot,\cdot]:\Lambda\times\Lambda'\rightarrow\mathbb{R}$ is a non-degenerate pairing. The pairing defines a lattice embedding $\Lambda'\subset\Hom(\Lambda,\mathbb{R})$ via the assignment $\lambda'\mapsto[\cdot,\lambda']$, and the torus itself is the compact abelian group $\Sigma=\Hom(\Lambda,\mathbb{R})/\Lambda'$. The integral structure refers to the lattice $\Hom(\Lambda,\mathbb{Z})\subset\Hom(\Lambda,\mathbb{R})$ in the universal cover of $\Sigma$, and is the tropical analogue of the complex structure on a complex torus.

Let $\Sigma_1=(\Lambda_1,\Lambda_1',[\cdot,\cdot]_1)$ and $\Sigma_2=(\Lambda_2,\Lambda_2',[\cdot,\cdot]_2)$ be integral tori. A \emph{homomorphism of integral tori} $f=(f_\#,f^\#):\Sigma_{1}\rightarrow\Sigma_2$ is given by a pair of maps $f_{\#}:\Lambda_1'\rightarrow\Lambda_2'$ and  $f^{\#}:\Lambda_2\rightarrow\Lambda_1$ satisfying the relation
\[
[f^{\#}(\lambda_2),\lambda_1']_1=[\lambda_2,f_{\#}(\lambda_1')]_2
\]
for all $\lambda_1'\in\Lambda_1'$ and $\lambda_2\in\Lambda_2$. This relation implies that the map $\Hom(\Lambda_1,\mathbb{R})\rightarrow  \Hom(\Lambda_2,\mathbb{R})$ dual to $f^{\#}$ restricts to $f_{\#}$ on $\Lambda'_1$, and hence descends to a group homomorphism $f:\Sigma_1\to\Sigma_2$. 

Let $f=(f_{\#},f^{\#}):\Sigma_1\rightarrow\Sigma_2$ be a homomorphism of integral tori $\Sigma_i=(\Lambda_i,\Lambda_i',[\cdot,\cdot]_i)$ for $i=1,2$. The connected component of the identity of the kernel of $f$, denoted by $(\Ker f)_0$, carries the structure of an integral torus, which we now recall. Let $K=(\co f^\#)^{tf}$ be the quotient of $\co f^{\#}$ by its torsion subgroup, and let $K'=\Ker f_\#$. It is easy to verify that the pairing $[\cdot,\cdot]_1$ induces a well-defined pairing $[\cdot,\cdot]_K : K \times  K' \rightarrow \mathbb{R}$, and that the natural maps $i_\#:K'\rightarrow \Lambda_1'$ and $i^\#:\Lambda_1\rightarrow K$ define an injective homomorphism $i=(i_\#,i^\#): (\Ker f)_0 \rightarrow \Sigma_1$ of integral tori.

Let $\Sigma=(\Lambda,\Lambda',[\cdot,\cdot])$ be an integral torus. A \emph{polarization} on $\Sigma$ is a map $\zeta:\Lambda'\rightarrow\Lambda$ having the property that the induced bilinear form
\[
(\cdot,\cdot):\Lambda'\times\Lambda'\rightarrow\mathbb{R},\quad(\lambda',\mu')=[\zeta(\lambda'),\mu']	
\]
is symmetric and positive definite. The polarization map $\zeta$ is necessarily injective, and is called \emph{principal} if it is bijective. The pair $(\Sigma,\zeta)$ is called a \emph{tropical polarized abelian variety}, and a \emph{tppav} if $\zeta$ is a principal polarization. 

Let $f=(f_\#,f^\#):\Sigma_1\rightarrow\Sigma_2$ be a homomorphism of integral tori, and assume that $f$ has finite kernel (equivalently, $f_{\#}$ is injective).	Given a polarization $\zeta_2:\Lambda_2'\rightarrow\Lambda_2$ on $\Sigma_2$, it is easy to verify that the map $\zeta_1=f^{\#} \circ \zeta_2 \circ f_{\#}:\Lambda_1'\rightarrow\Lambda_1$ defines an \emph{induced polarization} on $\Sigma_1$. The polarization induced by a principal polarization need not itself be principal.

Given a tropical polarized abelian variety $\Sigma=(\Lambda,\Lambda',[\cdot,\cdot])$ of dimension $n$, the associated bilinear form $(\cdot,\cdot)$ on $\Lambda'$ extends to an inner product on the vector space $V=\Hom(\Lambda,\mathbb{R})$. The volume of $\Sigma=V/\Lambda'$ with respect to this product is the volume of a fundamental parallelotope of $\Lambda'$, and is given by the Grammian determinant
\begin{equation}
\Vol^{2}(\Sigma)=\det(\lambda_i',\lambda_j'),
\label{eq:Grammian}
\end{equation}
where $\lambda_1',\ldots,\lambda_n'$  is any basis of $\Lambda'$.

\subsection{Jacobians and Pryms} We now recall how to construct the Jacobian of a metric graph and the Prym variety of a harmonic double cover as tppavs.

Let $\Gamma$ be a metric graph of genus $g$, let $C_0(\Ga,\mathbb{Z})=\mathbb{Z}^{V(\Ga)}$ and $C_{1}(\Ga,\mathbb{Z})=\mathbb{Z}^{E(\Ga)}$ be the groups of $0$-chains and $1$-chains, respectively (with respect to a choice of oriented model), and let $d$ be the simplicial boundary map
\[
d:C_1(\Ga,\mathbb{Z})\to C_0(\Ga,\mathbb{Z}),\quad d\left[\sum_{e \in E(\Ga)}n_{e}e\right]= \sum_{e \in E(\Ga)}n_{e}[t(e)-s(e)].
\]
The first simplicial homology group $H_1(\Ga,\ZZ)=\Ker d$ is a free abelian group of rank $g$. The \emph{edge length pairing} on $H_1(\Ga,\ZZ)$ is given by
\begin{equation}
[\cdot,\cdot]_{\Gamma} : H_{1}(\Ga,\mathbb{Z}) \times H_{1}(\Ga,\mathbb{Z}) \to \mathbb{R},\quad \Bigg[\sum_{\substack{e \in E(\Ga)}} a_{e }e, \sum_{\substack{e \in E(\Ga)}} b_{e }e \Bigg] = \sum_{\substack{e \in E(G)}} a_{e}b_{e}\ell(e).
\label{eq:edgelengthpairing}
\end{equation}
The \emph{Jacobian} $\Jac(\Gamma)$ of $\Ga$ is the dimension $g$ tppav $(\Lambda,\Lambda',[\cdot,\cdot]_{\Gamma})$ where $\Lambda=\Lambda'=H_{1}(\Ga,\mathbb{Z})$, $[\cdot,\cdot]_{\Gamma}$ is the edge length pairing, and the principal polarization $\zeta$ is the identity map on $H_1(\Ga,\mathbb{Z})$. 

\begin{remark} We note that the edge length pairing~\eqref{eq:edgelengthpairing} has a physical peculiarity: it is measured in units of edge lengths of $\Ga$, while the expected units for an inner product are lengths squared. As a consequence, the units of $\ell(e)$ double in dimension when we view the Jacobian variety $\Jac(\Ga)$ as a Riemannian manifold. For example, the Jacobian $\Jac(\Ga)$ of a circle $\Ga$ of circumference $L$ is also a circle, but of circumference $\sqrt{L}$, not $L$.
\label{rem:units}    
\end{remark}

Given a harmonic morphism $\phi:\tGa\to \Ga$ of metric graphs, we define the push-forward and pullback maps (again, with respect to appropriately chosen models)
\[
\phi_*:C_1(\tGa,\ZZ)\to C_1(\Ga,\ZZ),\quad \te\mapsto \phi(\te)
\]
and
\[
\phi^*:C_1(\Ga,\ZZ)\to C_1(\tGa,\ZZ),\quad e\mapsto \sum_{\te\in \phi^{-1}(e)}d_{\phi}(\te)\te.
\]
These maps commute with $d$ and descend to maps
\[
\phi_*:H_1(\tGa,\ZZ)\to H_1(\Ga,\ZZ),\quad \phi^*:H_1(\Ga,\ZZ)\to H_1(\tGa,\ZZ).
\]
The pair $\Nm_\#=\phi_*, \Nm^\#=\phi^*$ defines the surjective \emph{norm homomorphism} $\Nm:\Jac(\tGa)\to \Jac(\Ga)$.

We are interested in the kernel of the norm homomorphism $\Nm:\Jac(\tGa)\to \Jac(\Ga)$ when $\pi:\tGa\to \Ga$ is a harmonic double cover of metric graphs. In this case, the map $\pi^*:H_1(\Ga,\ZZ)\to H_1(\tGa,\ZZ)$ is explicitly given on edges by
\begin{equation}
    \pi^{*}(e)=
\begin{cases}
 2\te, & e\text{ is dilated with preimage }\te,  \\
 \te^++\te^-, & e\text{ is undilated with preimages }\te^{\pm}.
\end{cases}
 \label{eq:pb}
\end{equation}
Unwinding the definitions, the connected component of the identity of the kernel of $\Nm$ is the integral torus
\begin{equation*}
    (\Ker \Nm)_0 = \frac{\Ker \overline{\pi}: \Hom(H_{1}(\tGa,\mathbb{Z}),\mathbb{R} )\to \Hom(H_{1}(\Ga,\mathbb{Z}),\mathbb{R}) }{\Ker \pi_{*}: H_{1}(\tGa, \mathbb{Z}) \to H_{1}(\Ga, \mathbb{Z}) },
\end{equation*}
where $\overline{\pi}$ is the $\Hom$-dual of the map $\pi^{*}$. The principal polarization on $\Jac(\tGa)$ induces a polarization on $(\Ker \Nm)_0$, which is not principal in general. We show in Proposition~\ref{prop:pp} that there is a natural principal polarization on $(\Ker \Nm)_0$, whose type (compared to the induced polarization) depends on the number of connected components of the dilation subgraph $\Ga_{\dil}$. The corresponding tropical ppav is the \emph{Prym variety} $\Prym(\tGa/\Ga)$ of the double cover $\pi:\tGa\to \Ga$.

\subsection{Volume formulas and odd genus one decompositions} Finally, we recall how to compute the volumes of $\Jac(\Ga)$ for a metric graph $\Ga$ and $\Prym(\tGa/\Ga)$ for a free double cover $\tGa\to\Ga$. The formula for the volume of the tropical Prym variety of a dilated double cover is proved in the next section and is the principal result of this paper.

Let $\Ga$ be a metric graph of genus $g$. Since $\Jac(\Ga)$ is a Riemannian manifold of dimension $g$, one may expect the volume of $\Jac(\Ga)$ to be given by a homogeneous degree $g$ polynomial in the edge lengths $\ell(e)$ of $\Ga$. However, due to the dimensional peculiarity noted in Remark~\ref{rem:units}, it is the \emph{square} of the volume of $\Jac(\Ga)$ that is given by such a polynomial, with monomials corresponding to the complements of spanning trees of $\Ga$:

\begin{theorem}[Theorem 1.5 of~\cite{MR3264262}] Let $\Ga$ be a metric graph of genus $g$. The volume of the tropical Jacobian of $\Ga$ is given by
\begin{equation}
\Vol^2(\Jac(\Ga))=\sum_{F\subset E(\Ga)}\prod_{e\in F}\ell(e).
\label{eq:ABKS}
\end{equation}
The sum is taken over all $g$-element subsets $F\subset E(\Ga)$ with the property that $\Ga\backslash F$ is a tree.
\label{thm:ABKS}
\end{theorem}
It is elementary to verify that the sum on the right-hand side does not in fact depend on the choice of model of $\Ga$. We also note that if $F$ has $g$ elements, then $\Ga\backslash F$ is a tree if and only if it is connected.

An analogous formula for the volume of the tropical Prym variety of a free double cover is the principal result of~\cite{MR4382460}. Let $\pi:\tGa\to \Ga$ be a free double cover of metric graphs of genera $g(\tGa)=2g-1$ and $g(\Ga)=g$, respectively. Since $\Prym(\tGa/\Ga)$ has dimension $g-1$, we expect (see Remark~\ref{rem:units}) $\Vol^2(\Prym(\tGa/\Ga))$ to be given by a degree $g-1$ homogeneous polynomial in the edge lengths of $\Ga$. The monomials should correspond to certain $(g-1)$-element subsets of $E(\Ga)$, playing the same role that complements of spanning trees do for $\Jac(\Ga)$. The correct notion turns out to be the following.

\begin{definition} \label{def:ogodsfree} Let $\pi:\tGa\to \Ga$ be a free double cover of metric graphs of genera $g(\tGa)=2g-1$ and $g(\Ga)=g$, respectively. A set of $g-1$ edges $F\subset E(\Ga)$ is called an \emph{odd genus one decomposition}, or \emph{ogod}, if every connected component of $\Ga\backslash F$ has connected preimage in $\tGa$. The \emph{rank} $r(F)$ of an odd genus one decomposition is the number of connected components of $\Ga\backslash F$.
\end{definition}

An elementary calculation shows that for a set $F\subset E(\Ga)$ of $g-1$ edges with corresponding connected component decomposition $\Ga\backslash F=\Ga_1\cup \cdots\cup \Ga_k$, either $g(\Ga_i)=0$ for some $i$ or $g(\Ga_i)=1$ for all $i$. In the former case $\pi^{-1}(\Ga_i)$ is a trivial double cover (because $\pi_1(\Ga_i)=0$) and hence disconnected, while in the latter case $\pi^{-1}(\Ga_i)$ is connected if and only if the double cover $\pi^{-1}(\Ga_i)\to \Ga_i$ is given (under the Galois correspondence) by the odd (i.e.~notrivial) element of $\Hom(\pi_1(\Ga_i),\ZZ/2\ZZ)\simeq \ZZ/2\ZZ$. This explains our choice of terminology.

The volume of the tropical Prym variety is calculated as a sum over the odd genus one decompositions, with each monomial additionally weighted according to the rank.

\begin{theorem}[Theorem 3.4 in~\cite{MR4382460}] Let $\pi:\tGa\to \Ga$ be a free double cover of metric graphs of genera $2g-1$ and $g$, respectively. The volume of the tropical Prym variety of $\pi:\tGa\to \Ga$ is given by
\begin{equation}
\Vol^2(\Prym(\tGa/\Ga))=\sum_{F\subset E(\Ga)}4^{r(F)-1}\prod_{e\in F}\ell(e),
\label{eq:LZ}
\end{equation}
where the sum is taken over all odd genus one decompositions $F\subset E(\Ga)$.
\label{thm:LZ}
\end{theorem}

\begin{remark} Given a metric graph $\Ga$ of genus $g$, the complements of spanning trees are the bases of the \emph{cographic matroid} $\widetilde{\mathcal{M}}(\Ga)$. Similarly, a free double cover $\pi:\tGa\to \Ga$ determines (after choosing a spanning tree for $\Ga$) the structure of a \emph{signed graph} on $\Ga$, and the odd genus one decompositions are in fact the bases of the corresponding \emph{signed cographic matroid} $\widetilde{\mathcal{M}}(\tGa/\Ga)$ (see~\cite{zaslavsky1982signed}). Hence the sums on the right-hand sides of Equations~\eqref{eq:ABKS} and~\eqref{eq:LZ} are indexed by the bases of certain matroids naturally associated to $\Ga$ and $\pi:\tGa\to \Ga$, respectively. It turns out that the matroids $\widetilde{\mathcal{M}}(\Ga)$ and $\widetilde{\mathcal{M}}(\tGa/\Ga)$ play a fundamental role in the polyhedral geometry of $\Jac(\Ga)$ and $\Prym(\tGa/\Ga)$. These tppavs can be effectively described using respectively the Abel--Jacobi map $\Sym^g(\Ga)\to \Jac(\Ga)$ and the Abel--Prym map $\Sym^{g-1}(\tGa)\to \Prym(\tGa/\Ga)$. The independent sets of the matroids correspond to the cells of the symmetric product on which these maps have full rank. Hence the bases correspond to the top-dimensional cells that define polyhedral decompositions of the tppavs, and thus give geometric meaning to Equations~\eqref{eq:ABKS} and~\eqref{eq:LZ}.

\label{rem:matroids}
    
\end{remark}

\section{The volume formula}

In this section, we prove the main result of our paper, Theorem~\ref{thm:main}, which calculates the volume of the tropical Prym variety of a dilated double cover of metric graphs.

\subsection{Ogods for dilated double covers} Our first task is to extend Definition~\ref{def:ogodsfree} to dilated double covers. This turns out to be straightforward.

\begin{definition} Let $\pi:\tGa\to \Ga$ be a dilated double cover of a connected metric graph $\Ga$, and let $h=g(\tGa)-g(\Ga)$. A set $F\subset E(\Ga)$ of $h$ edges of $\Ga$ is called an \emph{ogod} if no edge in $F$ is dilated, and if each connected component of $\Ga\backslash F$ has connected preimage in $\tGa$. The \emph{rank} $r(F)$ of $F$ is the number of connected components of $\Ga\backslash F$.
\label{def:ogod}
\end{definition}

A connected component $\Ga_i$ of $\Ga\backslash F$ having a dilated vertex automatically has connected preimage in $\tGa$. If a connected component $\Ga_i$ has no dilated vertices, then $\pi^{-1}(\Ga_i)$ is connected only if $g(\Ga_i)\geq 1$, since a free double cover of a tree is trivial. To clarify exposition, and for future use, we give a more precise description of ogods for dilated double covers.

\begin{lemma} Let $\pi:\tGa\to \Ga$ be a dilated double cover of metric graphs, let $F\subset E(\Ga)$ be a set of $h=g(\tGa)-g(\Ga)$ undilated edges of $\Ga$, and let $\Ga\backslash F=\Ga_1\cup\cdots\cup \Ga_k$ be the decomposition of $\Ga\backslash F$ into connected components. Then $F$ is an ogod if and only if each $\Ga_i$ satisfies one of the following (mutually exclusive) conditions:
\begin{enumerate}
    \item $\Ga_i$ contains a unique connected component of the dilation subgraph $\Ga_{\dil}$, and the genus $g(\Ga_i)$ is equal to the genus of this component.\label{item:1}
    
    \item $\Ga_i$ has no dilated vertices or edges, $g(\Ga_i)=1$, and $\Ga_i$ has connected preimage in $\tGa$.\label{item:2}
\end{enumerate}
\label{lemma:ogods}
\end{lemma}

\begin{proof} It is clear that if each $\Ga_i$ is one of the above two types, then $F$ is an ogod. To prove the converse, we first assume that $\pi:\tGa\to \Ga$ is an edge-free cover, in which case the connected components of the dilation subgraph $\Ga_{\dil}$ are the dilated vertices. For a subgraph $\Ga_0\subset \Ga$, possibly disconnected, introduce the quantity
\[
\widetilde{g}(\Ga_0)=\#E(\Ga_0)-\#\{\mbox{undilated vertices of }\Ga_0\}.
\]
If $\Ga_0$ is connected, then
\[
\widetilde{g}(\Ga_0)=g(\Ga_0)-1+\#V(\Ga_0\cap \Ga_{\dil})\geq -1.
\]
Now let $F\subset E(\Ga)$ be an $h$-element set of undilated edges, and let $\Ga\backslash F=\Ga_1\cup\cdots\cup \Ga_k$ be the decomposition into connected components. It is clear that
\[
\widetilde{g}(\Ga\backslash F)=\sum_{i=1}^k \widetilde{g}(\Ga_i).
\]
On the other hand, we observe that
\[
h=g(\tGa)-g(\Ga)=\#E(\tGa)-\#V(\tGa)-(\#E(\Ga)-\#V(\Ga))=\#E(\Ga)-\#\{\mbox{undilated vertices of }\Ga\},
\]
therefore
\[
\widetilde{g}(\Ga\backslash F)=\#E(\tGa)-h-\#\{\mbox{undilated vertices of }\Ga\}=0.
\]
Since each $\widetilde{g}(\Ga_i)\geq -1$, it follows that either $\widetilde{g}(\Ga_i)=-1$ for some $i$ or $\widetilde{g}(\Ga_i)=0$ for all $i$. In the former case $F$ is not an ogod, because the component $\Ga_i$ with $\widetilde{g}(\Ga_i)=-1$ is a tree with no dilated vertices and hence $\pi^{-1}(\Ga_i)$ is disconnected. In the latter case, each $\Ga_i$ is either a tree with a unique dilated vertex, in which case it satisfies property~\eqref{item:1}, or a genus one graph with no dilated vertices, in which case it satisfies property~\eqref{item:2} if and only if it has connected preimage in $\tGa$. This proves the lemma for the edge-free double cover $\pi:\tGa\to \Ga$.

Now let $\pi:\tGa\to\Ga$ be a double cover with edge dilation. We consider the edge-free double cover $\pi':\tGa'\to\Ga'$ obtained by contracting each connected component of the dilation subgraph of $\Ga$ to a separate dilated vertex. The graphs $\Ga$ and $\Ga'$ have the same sets of undilated edges, and it is clear that the ogods of $\Ga$ and $\Ga'$ are in rank-preserving bijection. Now let $F\subset E(\Ga)$ be a set of $h$ undilated edges, let $\Ga\backslash F=\Ga_1\cup\cdots\cup \Ga_k$ be the decomposition into connected components, and let $\Ga'\backslash F=\Ga'_1\cup\cdots\cup \Ga'_k$ be the corresponding decomposition for $\Ga'$. If $\Ga_i$ has no dilation then $\Ga'_i=\Ga_i$, while if $\Ga_i$ has dilation, then it satisfies property~\eqref{item:1} if and only if $\Ga'_i$ is a tree with a unique dilated vertex. Hence each $\Ga_i$ satisfies property~\eqref{item:1} or~\eqref{item:2} if and only if $\Ga'_i$ does, in which case $F$ is an ogod.
\end{proof}

Lemma~\ref{lemma:ogods} shows that the term "odd genus one decomposition" makes sense for edge-free covers, if we view each dilated vertex as having intrinsic genus one. However, the terminology breaks down for covers with edge dilation, since now the genus of $\Ga_i$ is determined by the genus of the corresponding dilation subgraph, which may be arbitrary. For this reason, we henceforth use the term "ogod" instead of "odd genus one decomposition". We also note that ogods of dilated double covers also correspond to bases of an associated matroid (see Remark~\ref{rem:matroids}) on the set of undilated edges of $\Ga$, which we plan to investigate in future work.


We are now ready to state our main result.

\begin{theorem}
\label{thm:main}
Let $\pi:\tGa\to\Ga$ be a dilated double cover of metric graphs. The volume of the tropical Prym variety of $\pi:\tGa\to\Ga$ is given by
\begin{equation}
\Vol^2(\Prym(\tGa/\Ga))=2^{1-d(\tGa/\Ga)}\sum_{F\subset E(\Ga)} 4^{r(F)-1}\prod_{e\in F}\ell(e).
\label{eq:mainformula}
\end{equation}
The sum is taken over all $h$-element ogods of $E(\Ga)$ (see Definition~\ref{def:ogod}), where $h=g(\tGa)-g(\Ga)$ is the dimension of $\Prym(\tGa/\Ga)$ and $r(F)$ is the rank of the ogod, and $d(\tGa/\Ga)$ is the number of connected components of the dilation subgraph $\Ga_{\dil}$.

\end{theorem}

We first consider several examples. 

\begin{example} Let $\pi:\tGa\to\Ga$ be a double cover such that every vertex of $\Ga$ is dilated and no edge of $\Ga$ is dilated. In this case $g(\tGa)=2\#E(\Ga)-\#V(\Ga)+1$ so $h=\#E(\Ga)$, and the only ogod of the double cover $\pi:\tGa\to\Ga$ is all of $F=E(\Ga)$, with $r(F)=\#V(\Ga)$. Since $d(\tGa/\Ga)=\#V(\Ga)$, we see that the volume of the tropical Prym variety of $\pi:\tGa\to\Ga$ is equal to
\[
\Vol^2(\Prym(\tGa/\Ga))=2^{\#V(\Ga)-1}\prod_{e\in E(\Ga)}\ell(e).
\]

\end{example}

\begin{example} We consider the double cover $\pi:\tGa\to \Ga$ shown on Figure~\ref{fig:example1}. We have $g(\tGa)=6$ and $g(\Ga)=3$, so ogods are three-element subsets of the set $\{e_1,e_2,e_3,e_4,e_5\}$ of undilated edges. An ogod cannot contain both $e_1$ and $e_2$, since the left undilated vertex will then have disconnected preimage, and similarly with $e_4$ and $e_5$. This leaves a total of four ogods of the following ranks:
\[
r(\{e_1,e_3,e_4\})=3,\quad r(\{e_1,e_3,e_5\})=2,\quad r(\{e_2,e_3,e_4\})=4,\quad r(\{e_2,e_3,e_5\})=3.
\]
The dilation subgraph has $d(\tGa/\Ga)=2$ connected components, therefore 
\[
\Vol^2(\Prym(\tGa/\Ga))=8x_1x_3x_4+2x_1x_3x_5+32x_2x_3x_4+8x_2x_3x_5,\quad x_i=\ell(e_i).
\]
\end{example}

\begin{example} \label{ex:disc} Equation~\eqref{eq:mainformula} shows that the Prym variety of a dilated double cover behaves discontinuously under edge contractions that change the number $d(\tGa/\Ga)$ of connected components of the dilation subgraph. Indeed, consider the double cover $\pi:\tGa\to \Ga$ shown on the left hand side of Figure~\ref{fig:disc}. The double cover $\pi$ has two ogods, $\{e\}$ and $\{f\}$, with $r(\{e\})=2$ and $r(\{f\})=1$. The left vertex is dilated, so $d(\tGa/\Ga)=1$ and
\[
\Vol^2(\Prym(\tGa/\Ga))=4\ell(e)+\ell(f).
\]
The double cover $\pi':\tGa'\to \Ga'$ on the right hand side is obtained from $\pi$ by contracting the loop $f$, creating a second dilated vertex. The edge $e$ is the unique ogod and $r(\{e\})=2$, but now $d(\tGa'/\Ga')=2$. Hence
\[
\Vol^2(\Prym(\tGa'/\Ga'))=2\ell(e),
\]
which is not the limit of $\Vol^2(\Prym(\tGa/\Ga))$ as $\ell(f)\to 0$, as one may expect. This is problematic from a moduli-theoretic viewpoint, and suggests that the original definition of the Prym variety of a dilated double cover should be revisited.

We note that this phenomenon does not occur when deforming from a dilated double cover with connected dilation subgraph to a free double cover, since Equations~\eqref{eq:LZ} and~\eqref{eq:mainformula} agree when $d(\tGa/\Ga)=1$.

\begin{figure}
    \centering
\begin{tikzcd}
\begin{tikzpicture}

\draw[fill](0,0) circle(1mm);
\draw[thick] (0,0) -- (2,1);
\draw[thick] (0,0) -- (2,-1);
\draw[fill](2,1) circle(0.7mm);
\draw[fill](2,-1) circle(0.7mm);
\draw[thick] (2,1) .. controls (1.7,0.5) and (1.7,-0.5) .. (2,-1);
\draw[thick] (2,1) .. controls (2.3,0.5) and (2.3,-0.5) .. (2,-1);
\path[->] (1,-1) edge node[right]{$\pi$} (1,-2);

\begin{scope}[yshift=-3cm]

\draw[fill](0,0) circle(1mm);
\draw[thick] (0,0) -- (2,0);
\node[above] at (1,0) {$e$};
\node[right] at (3,0) {$f$};
\draw[fill](2,0) circle(0.7mm);
\draw[thick] (2.5,0) circle(5mm);

\end{scope}

\begin{scope}[xshift=5cm]
\draw[fill](0,0) circle(1mm);
\draw[thick] (0,0) .. controls (0.3,0.5) and (1.7,0.5) .. (2,0);
\draw[thick] (0,0) .. controls (0.3,-0.5) and (1.7,-0.5) .. (2,0);
\draw[fill](2,0) circle(1mm);
\path[->] (1,-1) edge node[right]{$\pi'$} (1,-2);
\end{scope}

\begin{scope}[xshift=5cm,yshift=-3cm]
\draw[fill](0,0) circle(1mm);
\draw[thick] (0,0) -- (2,0);
\draw[fill](2,0) circle(1mm);
\node[above] at (1,0) {$e$};

\end{scope}

\end{tikzpicture}
\end{tikzcd}
    \caption{Discontinuity of the Prym variety under edge contraction}
    \label{fig:disc}
\end{figure}
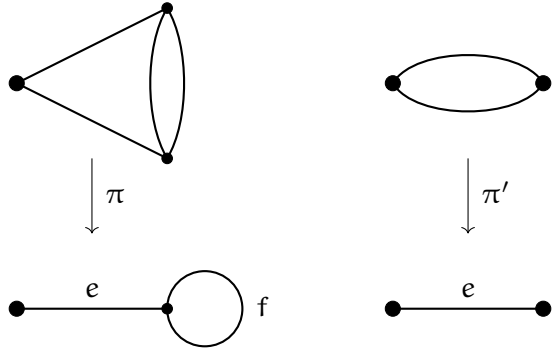

\end{example}

The proof of Theorem~\ref{thm:main} consists of two distinct parts. On one hand, we show that the polynomial on the right-hand side of Equation~\eqref{eq:mainformula} can be expressed in terms of the spanning trees of $\tGa$ and $\Ga$. The idea is to deform a dilated double cover to a free double cover by a series of edge contractions and de-contractions, and use Theorem~\ref{thm:LZ} as the base case. This part of the proof is purely graph-theoretic, and the main result is Theorem~\ref{thm:A}. On the other hand, we independently compute the relationship between the volumes of the tppavs $\Jac(\tGa)$, $\Jac(\Ga)$, and $\Prym(\tGa/\Ga)$ for a dilated double cover $\pi:\tGa\to \Ga$, by studying the action of the pushforward and pullback maps on the homology groups $H_1(\tGa,\ZZ)$ and $H_1(\Ga,\ZZ)$; the main result is Theorem~\ref{thm:B}. These homology calculations have appeared in~\cite{rohrle2022tropical}, sharpening and correcting the results of~\cite{MR4261102}, and we briefly reproduce them here.

\subsection{The volume polynomials} It is convenient to separate the right hand sides of Equations~\eqref{eq:ABKS} and~\eqref{eq:mainformula} into stand-alone definitions:

\begin{definition} Let $\Ga$ be a metric graph of genus $g$. The \emph{Jacobian polynomial} $J(\Ga)$ of $\Ga$ is the degree $g$ homogeneous polynomial in the edge lengths of $\Ga$ given by
\[
J(\Ga)=\sum_{C\subset E(\Ga)}\prod_{e\in C} \ell(e),
\]
where the sum is taken over all $g$-element subsets $C\subset E(\Ga)$ such that $\Ga\backslash C$ is a tree.
 \label{def:jp}   
\end{definition}

The following contraction-deletion formula for the Jacobian polynomial is elementary to verify, and we omit the proof.

\begin{lemma} Let $\Ga$ be a metric graph and let $e\in E(\Ga)$ be an edge of length $\ell(e)$. Let $\Ga_e$ and $\Ga^e$ be the graphs obtained by contracting and removing $e$, respectively. The Jacobian polynomial of $\Ga$ is expressed in terms of the Jacobian polynomials of $\Ga_e$ and $\Ga^e$ as follows:
\[
J(\Ga)=\begin{cases}
    \ell(e) J(\Ga_e), & e\mbox{ is a loop,} \\
    J(\Ga_e), & e\mbox{ is a bridge,} \\
    J(\Ga_e)+\ell(e)J(\Ga^e), & \mbox{otherwise}.
\end{cases}
\]
\label{lem:CD}
\end{lemma}

\begin{definition} Let $\pi:\tGa\to \Ga$ be a double cover of metric graphs, and let $h=g(\tGa)-g(\Ga)$. The \emph{Prym polynomial} $\Pr(\tGa/\Ga)$ of $\pi:\tGa\to \Ga$ is the degree $h$ homogeneous polynomial in the edge lengths of $\Ga$ given by
\begin{equation}
    \Pr(\tGa/\Ga)=\sum_{F\subset E(\Ga)}4^{r(F)-1}\prod_{e\in F}\ell(e),
\label{eq:prympolynomial}
\end{equation}
where the sum is taken over all ogods (see Definition~\ref{def:ogod}) and $r(F)$ is the rank of the ogod.
\label{def:prympolynomial}
\end{definition}

We now determine the relationship between the three polynomials associated with a harmonic double cover.

\begin{theorem}
\label{thm:A}
\noindent Let $\pi:\tGa\to \Ga$ be a double cover of metric graphs. Then
\begin{equation}
    J(\tGa)=2^{1-m_d(\tGa/\Ga)+n_d(\tGa/\Ga)-2d(\tGa/\Ga)}\Pr (\tGa/\Ga)J(\Ga).
      \label{eq:polynomials3}
\end{equation}
where $m_d(\tGa/\Ga)$, $n_d(\tGa/\Ga)$, and $d(\tGa/\Ga)$ denote respectively the number of edges, vertices, and connected components of the dilation subgraph $\Ga_{\dil}$. 

\end{theorem}

We note that for an edge $\te\in E(\tGa)$ we have 
\[
\ell(\te)=\begin{cases}\ell(\pi(\te)), & \pi(\te)\mbox{ is undilated,}\\
\ell(\pi(\te))/2, & \pi(\te)\mbox{ is dilated},
\end{cases}
\]
so we may indeed view $J(\tGa)$ as a polynomial in the edge lengths of $\Ga$. Of course, it is not a priori clear why $J(\tGa)$ should be divisible by $J(\Ga)$.

\begin{proof}[Proof of Theorem~\ref{thm:A}] We prove this result first for free double covers, then for edge-free double covers, and finally for double covers with edge dilation.

\medskip \noindent {\bf Free double covers.} The proof for a free double cover $\pi:\tGa\to \Ga$ follows directly from the results of~\cite{MR4382460}. Indeed, according to Equation~\eqref{eq:ABKS} we have
\[
\Vol^2(\Jac(\tGa))=J(\tGa),\quad \Vol^2(\Jac(\Ga))=J(\Ga).
\]
On the other hand, Equation~\eqref{eq:LZ} states that
\[
\Vol^2(\Prym(\tGa/\Ga))=\Pr(\tGa/\Ga),
\]
and the relationship between the three volumes is given by Proposition 3.6 of~\cite{MR4382460}:
\begin{equation}
\Vol^2(\Jac(\tGa))=2\Vol^2(\Prym(\tGa/\Ga))\Vol^2(\Jac(\Ga)).
\label{eq:freevolume}
\end{equation}
It follows that
\[
J(\tGa)=2\Pr(\tGa/\Ga)J(\Ga),
\]
which is Equation~\eqref{eq:polynomials3} for a free double cover, for which $m_d(\tGa/\Ga)=n_d(\tGa/\Ga)=d(\tGa/\Ga)=0$.

\begin{figure}
		\begin{tikzcd}
		\begin{tikzpicture}
		\draw[fill](0,0) circle(1mm) node[above=2pt]{$\tv$};
		
		\draw [thick] (0,0) -- (1.0,-0.5);
		\draw [thick] (0,0) -- (1,0.5);
		\draw [thick] (0,0) --  (-1,0.5);
              \draw [thick] (0,0) --  (-0.5,0.7);
		\draw [thick] (0,0) -- (-1,-0.5);      
                \draw[thick] (0,0) --( -0.5, -0.7);   
		\end{tikzpicture}
\arrow{d}{\bf{{\pi}}} 
		&
			\begin{tikzpicture}
		\draw[fill](0,0) circle(0.7mm) node[below=2pt] {$\tv^{-}$};
           \draw[fill](0,2) circle(0.7mm) node[above=2pt]{$\tv^{+}$};
           \draw [thick] (0,0) .. controls (0.4,1) and (0.4,1) .. (0,2);
      \draw [thick] (0,0) .. controls (-0.4,1) and (-0.4,1) .. (0,2);

            \draw [thick] (0,2) -- (1,2);
		\draw [thick] (0,2) --  (-1,2.5);
	      \draw [thick] (0,2) --(-1,1.5);
  
		\draw [thick] (0,0) -- (1.0,0.0);
		\draw [thick] (0,0) --  (-1,0.5);
		\draw [thick] (0,0) --(-1,-0.5);
            \node[right] at (0.2,1) {$\te^+$};
            \node[left] at (-0.2,1) {$\te^-$};

		\end{tikzpicture}
		\arrow{d}{\bf{{\pi'}}} 
		\\
		\begin{tikzpicture}
		\draw[fill](0,0) circle(1mm) node[below=2pt]{$v$};
		\draw [thick] (0,0) edge node[above] {} (1.0,0);
		\draw [thick] (0,0) edge node[above] {}(-1,0.5);
		\draw [thick] (0,0) edge node[below] {} (-1,-0.5);
		\end{tikzpicture}
	&
			\begin{tikzpicture}
		\draw[fill](0,0) circle(0.7mm) node[below=2pt]{$v$};
            \draw [thick] (0,0) .. controls (-0.5,0.3) and (-0.6,1) .. (0,1);
		\draw [thick] (0,0) .. controls (0.5,0.3) and (0.6,1) .. (0,1);
		\draw [thick] (0,0) edge node[above] {} (1.0,0);
		\draw [thick] (0,0) edge node[above] {}(-1,0.5);
		\draw [thick] (0,0) edge node[below] {} (-1,-0.5);
            \node[right] at (0.3,0.6) {$e$};
		\end{tikzpicture}
		\end{tikzcd}
		\\
  \caption{Resolving a dilated vertex by adding a loop.}
\label{fig:addloop}
\end{figure}
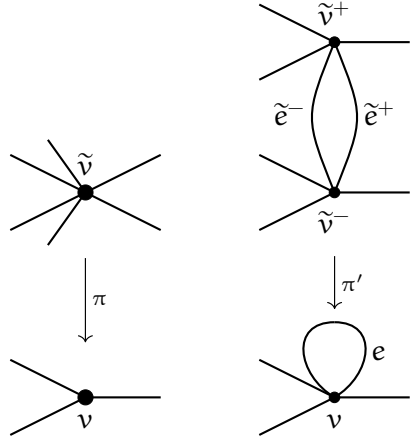

\medskip \noindent {\bf Edge-free double covers.} We now prove the theorem for edge-free double covers by induction on the number of dilated vertices. For such a cover $\pi:\tGa\to \Ga$, we have $n_d(\tGa/\Ga)=d(\tGa/\Ga)$ and $m_d(\tGa/\Ga)=0$. Consider an edge-free cover $\pi:\tGa\to \Ga$ with a dilated vertex $v\in V(\Ga)$, and let $\tv=\pi^{-1}(v)$. We consider the double cover $\pi':\tGa'\to \Ga'$ obtained by resolving the dilated vertex $v$ into an undilated vertex by a loop attachment (see Figure~\ref{fig:addloop}). Specifically, $\Ga'$ consists of $\Ga$ with a loop $e$ of length $\ell(e)$ attached to $v$, while $\tGa'$ consists of $\tGa$ with the vertex $\tv$ replaced by a pair of vertices $\tv^{\pm}$ connected by two edges $\te^{\pm}$. For each edge $f\in E(\Ga)$ rooted at $v$ there are two edges $\tf^{\pm}\in E(\tGa)$ rooted at $\tv$, and we root one at each of the $\tv^{\pm}\in V(\tGa')$ arbitrarily. The map $\pi'$ sends $\tv^{\pm}$ to $v$, $\te^{\pm}$ to $e$, and is equal to $\pi$ on the rest of $\tGa'$. The vertex $v$ on the resulting harmonic double cover $\pi':\tGa'\to\Ga'$ is now undilated, hence
\begin{equation}
n_d(\tGa/\Ga)=d(\tGa/\Ga)=n_d(\tGa'/\Ga')+1=d(\tGa'/\Ga')+1,\quad
m_d(\tGa/\Ga)=m_d(\tGa'/\Ga')=0,
\label{eq:contr1}
\end{equation}
and we assume by induction that Equation~\eqref{eq:polynomials3} holds for $\pi':\tGa'\to \Ga'$.

We now compare the Jacobian and Prym polynomials of the two covers $\pi':\tGa'\to \Ga'$ and $\pi:\tGa\to \Ga$. First, we note that $g(\Ga')=g(\Ga)+1$. Since $e$ is a loop, it lies in the complement of each spanning tree, and therefore
\begin{equation}
J(\Ga')=\ell(e)J(\Ga).
\label{eq:contr2}
\end{equation}

Similarly, we have $g(\tGa')=g(\tGa)+1$, and we evaluate $J(\tGa')$ in terms of $J(\tGa)$. We classify complements of spanning trees $C'\subset E(\tGa')$ according to whether they contain the edges $\te^{\pm}$:
\begin{enumerate}
    \item $\te^{\pm}\in C'$. There is a unique path from $\tv^+$ to $\tv^-$ along the spanning tree $\tGa'\backslash C'$ of $\tGa'$, which corresponds to a closed loop in $\tGa$ starting and ending at $\tv$. Hence it follows that the corresponding subset $C=C'\backslash \{\te^{\pm}\}\subset E(\tGa)$ is not the complement of a spanning tree of $\tGa$. Instead, $C$ is the complement of a unique spanning tree in the graph $\tGa'_0$ obtained from $\tGa'$ by deleting the edges $\te^{\pm}$, and every such complement is obtained in this way.

    \item $\te^+\in C'$ and $\te^-\notin C'$. The spanning tree $T'=\tGa'\backslash C'$ contains the edge $\te^-$, and $T=T'\backslash \{\te^-\}$ is a spanning tree of $\tGa$ with complementary edge set $C=C'\backslash \{\te^+\}\subset E(\tGa)$. Conversely, if $C\subset E(\tGa)$ is the complement of a spanning tree $T=\tGa\backslash C$ of $\tGa$, then $C'=C\cup \te^+\subset E(\tGa')$ is the complement of a spanning tree $T'=T\cup \{\te^-\}$ of $\tGa'$.

    \item $\te^-\in C'$ and $\te^+\notin C'$. This case is symmetric to the one above: $C=C'\backslash \{\te^-\}$ is the complement of a spanning tree of $\tGa$, and every such complement is obtained in this way.

    \item $\te^{\pm}\notin C'$. This is not possible, since a spanning tree of $\tGa'$ may not contain both edges $\te^{\pm}$. 
    
\end{enumerate}
Expressing the sum that defines $J(\tGa')$ according to these four types, we see that
\begin{equation}
J(\tGa')=2\ell(e)J(\tGa)+\ell(e)^2 J(\tGa'_0).
\label{eq:contr3}
\end{equation}

Finally, we compare $\Pr(\tGa'/\Ga')$ and $\Pr(\tGa/\Ga)$. Since $g(\tGa')-g(\Ga')=g(\tGa)-g(\Ga)$, these polynomials have the same degree. Let $F\subset E(\Ga)$ be an ogod of the double cover $\pi:\tGa\to \Ga$. The connected component of $\Ga\backslash F$ containing $v$ has connected pre-image in $\tGa$ because $v$ is dilated, and this remains true when we undilate $v$ and replace $\tv$ with $\tv^{\pm}$. Hence $F\subset E(\Ga')$ is also an ogod of the double cover $\pi':\tGa'\to \Ga$ of the same rank. Conversely, if $F'\subset E(\Ga')$ is an ogod of the double cover $\pi':\tGa'\to \Ga$ and $e\notin F'$, then $F'\subset E(\Ga)$, and the connected component of $\Ga'\backslash F'$ containing $e$ (and having connected pre-image in $\tGa'$) corresponds to a connected component of $\Ga\backslash F'$ containing $v$ (which necessarily has connected pre-image). It follows that there is a rank-preserving bijection between the ogods $F\subset E(\Ga)$ of the double cover $\pi:\tGa\to \Ga$ and the ogods $F'\subset E(\Ga')$ of the double cover $\pi':\tGa'\to \Ga'$ not containing $e$, and therefore
\begin{equation}
\Pr(\tGa'/\Ga')=\Pr(\tGa/\Ga)+\ell(e)Q,
\label{eq:contr4}
\end{equation}
where the term $Q$ is irrelevant to us.

We now put everything together. By induction, Equation~\eqref{eq:polynomials3} holds for the double cover $\pi':\tGa'\to \Ga'$. Plugging in Equations~\eqref{eq:contr2},~\eqref{eq:contr3}, ~\eqref{eq:contr4}, taking the linear in $\ell(e)$ term, and using~\eqref{eq:contr1}, we see that Equation~\eqref{eq:polynomials3} holds for the double cover $\pi:\tGa\to \Ga$.

\medskip \noindent {\bf Double covers with edge dilation.} Finally, we prove the theorem for arbitrary dilated double covers by induction on the number of dilated edges, the base case being that of edge-free double covers. Let $\pi:\tGa\to \Ga$ be a dilated double cover with a dilated edge $e\in E(\Ga)$ of length $\ell(e)$, and let $\te=\pi^{-1}(e)$ be its preimage of length $\ell(\te)=\ell(e)/2$. We contract the edges $e$ and $\te$ to obtain a dilated double cover $\pi_e:\tGa_{\te}\to \Ga_e$ with
\begin{equation}
m_d(\tGa_{\te}/\Ga_e)=m_d(\tGa/\Ga)-1,\quad
d(\tGa_{\te}/\Ga_e)=d(\tGa/\Ga),
\label{eq:decontr1}
\end{equation}
and we assume by induction that Equation~\eqref{eq:polynomials3} holds for $\pi_e:\tGa_{\te}\to \Ga_e$.

It is clear that $g(\tGa_{\te})-g(\Ga_e)=g(\tGa)-g(\Ga)$, so $\Pr(\tGa_{\te}/\Ga_e)$ and $\Pr(\tGa/\Ga)$ have the same degrees. Since ogods do not contain dilated edges, and the dilation subgraphs of $\Ga$ and $\Ga_e$ have the same sets of connected components, we see that there is a rank-preserving bijection between the ogods of the double covers $\pi_e:\tGa_{\te}\to \Ga_e$ and $\pi:\tGa\to \Ga$. Therefore
\begin{equation}
\Pr(\tGa_{\te}/\Ga_e)=\Pr(\tGa/\Ga).
\label{eq:decontr2}
\end{equation}

We now consider the edge types of $e$ and $\te$ and apply Lemma~\ref{lem:CD}.

\begin{enumerate}
    \item If $e\in E(\Ga)$ is a loop, then $\te\in E(\tGa)$ is also a loop (of half the length) because the root vertex of $e$ is dilated and hence has a unique preimage at which both ends of $\te$ are rooted. Then by Lemma~\ref{lem:CD} we have
\begin{equation}
J(\Ga)=\ell(e) J(\Ga_e),\quad J(\tGa)=\frac{\ell(e)}{2}J(\tGa_{\te}).
\label{eq:decontr3}
\end{equation}
Plugging Equations~\eqref{eq:decontr1},~\eqref{eq:decontr2}, and~\eqref{eq:decontr3} into Equation~\eqref{eq:polynomials3} and noting that $n_d(\tGa_{\te}/\Ga_e)=n_d(\tGa/\Ga)$, we see that Equation~\eqref{eq:polynomials3} holds for the double cover $\pi:\tGa\to \Ga$.

\item If $e\in E(\Ga)$ is a bridge, then $\te\in E(\tGa)$ is also a bridge because $\pi^{-1}(e)=\te$. By Lemma~\ref{lem:CD} we have
\begin{equation}
J(\Ga)=J(\Ga_e),\quad J(\tGa)=J(\tGa_{\te}),
\label{eq:decontr4}
\end{equation}
and in this case $n_d(\tGa_{\te}/\Ga_e)=n_d(\tGa/\Ga)-1$, since contracting $e$ joins the dilated end vertices of $e$ into a single dilated vertex. Plugging this and Equations~\eqref{eq:decontr1},~\eqref{eq:decontr2} and~\eqref{eq:decontr4} into Equation~\eqref{eq:polynomials3}, we see that Equation~\eqref{eq:polynomials3} holds for the double cover $\pi:\tGa\to \Ga$.

\item Finally, suppose that $e\in E(\Ga)$ is neither a bridge nor a loop. There exists a path in $\Ga$ between the (dilated) end vertices of $\te$ that bypasses $e$, and this path lifts to a path in $\tGa$ connecting the end vertices of $\te$ and bypassing $\te$. Hence $\te\in E(\tGa)$ is neither a bridge nor a loop, and by Lemma~\ref{lem:CD} we have
\begin{equation}
J(\Ga)=J(\Ga_e)+\ell(e) J(\Ga^e), \quad J(\tGa)=J(\tGa_{\te})+\frac{\ell(e)}{2}J(\tGa^{\te}),
\label{eq:decontr5}
\end{equation}
and $n_d(\tGa_{\te}/\Ga_e)=n_d(\tGa/\Ga)-1$ as in the previous case.

We need to additionally consider the harmonic double cover $\pi^e:\tGa^{\te}\to \Ga^e$ obtained by deleting the edges $e$ and $\te$. Since $\Ga^e$ has one fewer dilated edge than $\Ga$, we assume by induction that Equation~\eqref{eq:polynomials3} holds for the double cover $\pi^e:\tGa^{\te}\to \Ga^e$, and we have
\begin{equation}
m_d(\tGa^{\te}/\Ga^e)=m_d(\tGa/\Ga)-1,\quad
n_d(\tGa^{\te}/\Ga^e)=n_d(\tGa/\Ga).
\label{eq:decontr6}
\end{equation}
Since $e$ is dilated, there is again a bijection between the ogods of the double covers $\pi^e:\tGa^{\te}\to \Ga^e$ and $\pi:\tGa\to \Ga$. However, the ranks of the ogods may be different, since removing $e$ may increase the number of connected components of the dilation subgraph. We consider two subcases:

\begin{enumerate}
    \item The edge $e$ is not a bridge edge of the dilation subgraph $\Ga_{\dil}$, in which case $d(\tGa^{\te}/\Ga^e)=d(\tGa/\Ga)$. Given an ogod $F\subset E(\Ga)$ of $\Ga$, the connected component $\Ga_i$ of $\Ga\backslash F$ containing $e$ is not disconnected by removing $e$. It follows that the ranks of $F$ as an ogod on $\Ga$ and $\Ga^e$ agree, and hence $\Pr(\tGa/\Ga)=\Pr(\tGa^{\te}/\Ga^e)$. Plugging this and Equations~\eqref{eq:decontr1},~\eqref{eq:decontr2},~\eqref{eq:decontr5}, and~\eqref{eq:decontr6} into Equation~\eqref{eq:polynomials3}, we see that Equation~\eqref{eq:polynomials3} holds for the double cover $\pi:\tGa\to \Ga$.

    \item The edge $e$ is a bridge edge of the dilation subgraph $\Ga_{\dil}$, so that $d(\tGa^{\te}/\Ga^e)=d(\tGa/\Ga)+1$. Let $F\subset E(\Ga)$ be an ogod of $\Ga$, and let $\Ga_i$ be the connected component of $\Ga\backslash F$ containing $e$. By Lemma~\ref{lemma:ogods}, the dilation subgraph of $\Ga_i$ is connected and has the same genus as $\Ga_i$. It follows that $e$ is in fact a bridge edge of $\Ga_i$ itself, not just its dilation subgraph. Hence $\Ga^e\backslash F$ has one more connected component than $\Ga\backslash F$, so the rank of $F$ as an ogod of $\Ga^e$ is one greater than its rank as an ogod of $\Ga$, and therefore $\Pr(\tGa/\Ga)=\Pr(\tGa^{\te}/\Ga^e)/4$.  Plugging this and the remaining formulas into Equation~\eqref{eq:polynomials3}, we see that Equation~\eqref{eq:polynomials3} holds for the double cover $\pi:\tGa\to \Ga$.

\end{enumerate}

\end{enumerate}

\end{proof}

\subsection{The volumes of the tppavs} We now compute the relationship between the volumes of the three tppavs associated to a dilated double cover $\pi:\tGa\to\Ga$. The main result is the following theorem. 

\begin{theorem} 
Let $\pi:\tGa\to \Ga$ be a dilated double cover of metric graphs. The volume of the tropical Prym variety of $\pi$ is given by
\begin{equation}
    \Vol^2(\Prym(\tGa/\Ga)) = 2^{m_d(\tGa/\Ga) -n_d(\tGa/\Ga) + d(\tGa/\Ga)} \frac{\Vol^2(\Jac(\tGa))}{\Vol^2(\Jac(\Ga))},
\label{eq:dilatedvolume}
\end{equation}
where $m_d(\tGa/\Ga)$, $n_d(\tGa/\Ga)$, and $d(\tGa/\Ga)$ are respectively the numbers of edges, vertices, and connected components of the dilation subgraph $\Ga_{\dil}$.
\label{thm:B}
\end{theorem}

Before giving the proof, we consider an elementary example.

\begin{example} Let $\Ga$ be a metric graph of genus $g$, and let $\pi:\tGa\to\Ga$ be the double cover such that $\Ga_{\dil}=\Ga$, so that $\pi$ is a factor two isometry. Since $g(\tGa)=g(\Ga)$ the Prym variety $\Prym(\tGa/\Ga)$ is a point, and its (zero-dimensional) volume is formally equal to one. On the other hand, the exponent in the right hand side of Equation~\eqref{eq:dilatedvolume} is the genus of $\Ga$, so we see that
\[
\Vol^2(\Jac(\tGa))=2^{-g(\Ga)}\Vol^2(\Jac(\Ga)).
\]
This clearly agrees with Theorem~\ref{thm:ABKS}, since each edge has half the length in $\tGa$ as in $\Ga$, and thus the Jacobians of $\tGa$ and $\Ga$ differ by scaling by a factor of $2$.
    
\end{example}

The principal technical result required for the proof is Proposition~\ref{prop:dilatedbasis}, which calculates the pushforward, pullback, and involution maps
\[
\pi_*:H_1(\tGa,\ZZ)\to H_1(\Ga,\ZZ), \quad \pi_*:H_1(\tGa,\ZZ)\to H_1(\Ga,\ZZ), \quad \iota_*:H_1(\tGa,\ZZ)\to H_1(\tGa,\ZZ)
\]
in terms of explicit bases of $H_1(\tGa,\ZZ)$ and $H_1(\Ga,\ZZ)$. This result recently appeared in~\cite{rohrle2022tropical}, improving and correcting earlier results in~\cite{MR4261102}, and we restate it here for convenience. We then calculate the volumes of the tppavs using Equation~\eqref{eq:Grammian}. We also note that, unlike Theorem~\ref{thm:A}, the relationship between the volumes in the case of a free double cover is given by a different formula (Equation~\eqref{eq:freevolume}, which differs from Equation~\eqref{eq:dilatedvolume} by a factor of two). Morally, this is due to the fact that the kernel of the norm map $\Nm:\Jac(\tGa)\to \Jac(\Ga)$ has two connected components if $\pi:\tGa\to \Ga$ is free and one if it is dilated (see Theorem 1.5.7 in~\cite{MR4261102}).




Let $\pi:\tGa\to \Ga$ be a dilated double cover, and introduce the invariants
\begin{equation} \label{eq:ABC}
A = g(\Ga)-m_d+n_d-d, \quad
B = d-1, \quad
C = m_d-n_d+d,
\end{equation}
where $m_d$, $n_d$, and $d$ denote respectively the number of edges, vertices, and connected components of the dilation subgraph $\Ga_{\dil}$. We note that
\[
A+B=g(\Ga)-m_d+n_d-1=|E(\Ga)|-|V(\Ga)|-m_d+n_d=g(\tGa)-g(\Ga)
\]
is the dimension of the Prym variety of the double cover $\pi:\tGa\to \Ga$. We explicitly describe the induced maps on the homology groups:

\begin{proposition}[Proposition 4.20 in~\cite{rohrle2022tropical}] \label{prop:dilatedbasis} Let $\pi:\tGa\to\Ga$ be a dilated double cover of metric graphs. There exists a basis $\al_1,\ldots,\al_A$, $\ga_1,\ldots,\ga_C$
    of $H_1(\Ga,\ZZ)$ and a basis $\tal^{\pm}_1,\ldots,\tal^{\pm}_A$,  $\tbe_1,\ldots,\tbe_B$, $\tga_1,\ldots,\tga_C$
    of $H_1(\tGa,\ZZ)$ such that
    \begin{align*}
        \iota_*(\tal^{\pm}_i) &= \tal^{\mp}_i,
        & \pi_*(\tal^{\pm}_i) &= \al_i,
        & \pi^*(\al_i) &= \tal^+_i+\tal^-_i,
        & i&=1,\ldots,A, \\
        \iota_*(\tbe_j) &= -\tbe_j,
        & \pi_*(\tbe_j) &= 0,
        &&& j&=1,\ldots,B,\\
        \iota_*(\tga_k) &= \tga_k,
        & \pi_*(\tga_k) &= \ga_k,
        & \pi^*(\ga_k) &= 2\tga_k,
        & k&=1,\ldots,C.
    \end{align*}
\end{proposition}

We now show how to define a principal polarization on the tropical Prym variety $\Prym(\tGa/\Ga)$ associated to a dilated double cover $\pi:\tGa\to \Ga$. Recall that the underlying integral torus of $\Prym(\tGa/\Ga)$ is given by the triple $(K,K',[\cdot,\cdot]_P)$, where $K=(\co \pi^*)^{tf}$, $K'=\Ker \pi_*$, and the intersection pairing $[\cdot,\cdot]_P:K'\times K\to \RR$ is induced from the pairing $H_1(\tGa,\ZZ)\times H_1(\tGa,\ZZ)\to \RR$ on $\tGa$. The principal polarization $\xi=\Id:H_1(\tGa, \ZZ)\to H_1(\tGa, \ZZ)$ induces a polarization $\xi|_P:K'\to K$, which is not principal in general. The structure of this polarization was computed in~\cite{rohrle2022tropical}.

\begin{proposition} [Proposition 4.21 in~\cite{rohrle2022tropical}] \label{prop:pp} There exists a principal polarization $\zeta:K'\to K$ on $\Prym(\tGa/\Ga)$ with respect to which the induced polarization $\xi|_P$ has type (i.e.~Smith normal form) $(1,\ldots,1,2,\ldots,2)$, where the number of $1$'s and $2$'s is equal to $B$ and $A$, respectively. 
\end{proposition}



We are now ready to prove Theorem~\ref{thm:B}.

\begin{proof}[Proof of Theorem~\ref{thm:B}]

Let $\tcalB=\{\tal^{\pm}_i,\tbe_j,\tga_k\}$ be the $\ZZ$-basis of $H_1(\tGa,\ZZ)$ constructed in Proposition~\ref{prop:dilatedbasis}, then by Equation~\eqref{eq:Grammian} we have
\[
\Vol^2(\Jac(\tGa))=\Gram(\tcalB)_{\tGa},
\]
where the subscript $\tGa$ is there to remind us that the entries in the Gram determinant are computed using the inner product $(\cdot,\cdot)_{\tGa}$. Introduce the following alternative $\QQ$-basis $\tcalB'$ of $H_1(\tGa,\ZZ)$:
\[
\tcalB'=\tcalB'_1\cup\tcalB'_2,\quad \tcalB'_1=\{\tal^+_i+\tal^-_i,2\tga_k\},\quad \tcalB'_2=\{\tal^+_i-\tal^-_i,\tbe_j\}.
\]
The change-of-basis matrix from $\tcalB$ to $\tcalB'$ has determinant $\pm 2^{-A-C}$ (see Equation~\eqref{eq:ABC}), hence
\[
\Gram(\tcalB)_{\tGa}=2^{-2A-2C}\Gram(\tcalB')_{\tGa}.
\]
Now let $\tde_1\in \tcalB'_1$ and $\tde_2\in \tcalB'_2$ be two cycles. By Proposition~\ref{prop:dilatedbasis} we know that $\iota_*(\tde_1)=\tde_1$ and $\iota_*(\tde_2)=-\tde_2$. Since $\iota_*$ preserves the inner product $(\cdot,\cdot)_{\tGa}$, we have
\[
(\tde_1,\tde_2)_{\tGa}=(\iota_*(\tde_1),\iota_*(\tde_2))_{\tGa}=-(\tde_1,\tde_2)_{\tGa}.
\]
Hence the elements of $\tcalB'_1$ and $\tcalB'_2$ are orthogonal to one another, and therefore
\[
\Gram(\tcalB')_{\tGa}=\Gram(\tcalB'_1)_{\tGa}\Gram(\tcalB'_2)_{\tGa}.
\]
Now let $\calB=\{\al_i,\ga_k\}$ be the basis of $H_1(\Ga,\ZZ)$ given in Proposition~\ref{prop:dilatedbasis}, then $\Gram(\calB)_{\Ga}$ computes $\Vol^2(\Jac(\Ga))$ by Equation~\eqref{eq:Grammian}. On the other hand, $\pi^*(\calB)=\tcalB'_1$, and for any $\de_1,\de_2\in H_1(G,\ZZ)$ by Equation~\eqref{eq:pb} we have
\[
(\pi^*(\de_1),\pi^*(\de_2))_{\tGa}=2(\de_1,\de_2)_{\Ga},
\]
implying that
\[
\Gram(\tcalB'_1)_{\tGa}=2^{A+C}\Gram(\calB)_{\Ga}=2^{A+C}\Vol^2(\Jac(\Ga)).
\]
Finally, we note that $\tcalB'_2$ is a basis for $K'=\Ker \pi_*$, so the Gram determinant $\Gram(\tcalB'_2)_{\tGa}$ computes the square of the volume of the tropical Prym variety, but with respect to the polarization $\xi|_P$ induced from $\Jac(\tGa)$. By 
Proposition~\ref{prop:pp}, the volume with respect to the intrinsic principal polarization is obtained by re-scaling as follows:
\[
\Vol^2(\Prym(\tGa/\Ga))=\Gram(\tcalB'_2)_P=2^{-A}\Gram(\tcalB'_2)_{\tGa}.
\]
Putting everything together, we see that
\[
\Vol^2(\Jac(\tGa))=\Gram(\tcalB)_{\tGa}=2^{-2A-2C}\Gram(\tcalB')_{\tGa}=2^{-2A-2C}\Gram(\tcalB'_1)_{\tGa}\Gram(\tcalB'_2)_{\tGa}=
\]
\[
=2^{-2A-2C}\cdot 2^{A+C}\Gram(\calB)_{\Ga}2^A \Gram(\tcalB)_P=2^{-C}\Vol^2(\Jac(\Ga))\Vol^2(\Prym(\tGa/\Ga)),
\]
which completes the proof.

\end{proof}

\subsection{Proof of the main theorem} We are now ready to put everything together.

\begin{proof}[Proof of Theorem~\ref{thm:main}] Let $\pi:\tGa\to \Ga$ be a dilated double cover of metric graphs, and let $m_d$, $n_d$, and $d$ denote respectively the number of edges, vertices, and connected components of the dilation subgraph $\Ga_{\dil}$. From Theorems~\ref{thm:B},~\ref{thm:ABKS}, and~\ref{thm:A}, respectively, we have
\[
\Vol^2(\Prym(\tGa/\Ga))=2^{m_d-n_d+d}\frac{\Vol^2(\Jac(\tGa))}{\Vol^2(\Jac(\Ga))}
= 2^{m_d-n_d+d} \frac{J(\tGa)}{J(\tGa)}
= 2^{1-d}\Pr(\tGa/\Ga),
\]
which by Definition ~\ref{def:prympolynomial} is the right hand side of~\eqref{eq:mainformula}.

\end{proof}

\section*{}
\bibliographystyle{alpha}
\addcontentsline{toc}{section}{References}
\bibliography{volume.bib}
\end{document}